\documentclass[11pt,twoside]{article}

\usepackage{fullpage}

\usepackage{amssymb}
\usepackage{amsmath,amsthm,amsfonts,dsfont,bm}
\usepackage{url}
\usepackage{rotating}
\usepackage{multirow}
\usepackage{color}
\usepackage{xcolor}
\usepackage{enumitem}

\usepackage{hyperref}

\usepackage[showonlyrefs]{mathtools}

\hypersetup{
  colorlinks   = true, 
  urlcolor     = blue, 
  linkcolor    = blue, 
  citecolor   = red 
}

\newtheorem{theorem}{Theorem}

\newtheorem{lemma}[theorem]{Lemma}

\newtheorem{remark}{Remark}

\numberwithin{equation}{section}
\numberwithin{theorem}{section}

\renewcommand{\P}{\operatorname{\mathbb{P}}}

\newcommand{\E}{\operatorname{\mathbb{E}}}
\newcommand{\N}{\mathbb{N}}
\newcommand{\Z}{\mathbb{Z}}

\newcommand{\R}{\mathbb{R}}

\usepackage[mathscr]{euscript}
\DeclareSymbolFont{rsfs}{U}{rsfs}{m}{n}
\DeclareSymbolFontAlphabet{\mathscrsfs}{rsfs}

\newcommand{\one}{\mathbf{1}}
\newcommand{\bsigma}{\bm{\sigma}}
\newcommand{\eps}{\varepsilon}
\newcommand{\rmd}{\mathrm{d}}

\newcommand{\cR}{\mathcal{R}}

\newcommand{\cC}{\mathcal{C}}
\newcommand{\bg}{\bm{g}}

\newcommand{\bG}{\bm{G}}
\def\Sreg{S_{\textup{reg}}}

\def\mA{\mathcal{A}}
\def\lbq{\underline{q}}
\def\ubq{\overline{q}}

\def\salg{\mbox{\rm\tiny ALG}}

\def\pl{\mbox{\rm\tiny pl}}
\def\rd{\mbox{\rm\tiny rd}}

\def\LHS{\mbox{\small\rm LHS}}

\begin{document}

\title{
Near-optimal shattering in the Ising pure $p$-spin and \\ rarity of solutions returned by stable algorithms
}

\author{Ahmed El Alaoui\thanks{Department of Statistics and Data Science, Cornell University. Email: elalaoui@cornell.edu}}

\date{}
\maketitle

\vspace*{-.3cm} 
\begin{abstract}
We show that in the Ising pure $p$-spin model of spin glasses, shattering takes place at all inverse temperatures $\beta \in (\sqrt{(2 \log p)/p}, \sqrt{2\log 2})$ when $p$ is sufficiently large as a function of $\beta$. Of special interest is the lower boundary of this interval which matches the large $p$-asymptotics of the inverse temperature marking the hypothetical dynamical transition predicted in statistical physics. We show this as a consequence of a `soft' version of the overlap gap property which asserts the existence of a distance gap of points of typical energy from a typical sample from the Gibbs measure. We further show that this latter property implies that stable algorithms seeking to return a point of at least typical energy are confined to an exponentially rare subset of that super-level set, provided that their success probability is not vanishingly small.          
\end{abstract}

 \section{Introduction and main result}  
We consider the mixed $p$-spin model on the binary hypercube. Let $(g_{i_1,\cdots,i_k})_{1 \le i_1,\cdots,i_k \le N, k \ge 2}$ be a collection of i.i.d.\ $N(0,1)$ random variables which we denote by $\bG$, and fix a collection of scalars $(\gamma_{k})_{k=2}^P$. We consider the random Hamiltonian 
\begin{align}\label{eq:hamiltonian}
H_N(\bsigma) =  \sum_{k=2}^{P} \frac{\gamma_k}{N^{(k-1)/2}}  \sum_{1\le i_1 ,\cdots , i_k\le N} g_{i_1,\cdots,i_k}\sigma_{i_1}\cdots\sigma_{i_k} \,,~~~ 
\bsigma \in \{-1,+1\}^N\,.
\end{align}

 Let $\xi(x) := \sum_{k= 2}^P \gamma_k^2 x^k$  be the mixture function of $H_N$; we have $\E[H_N(\bsigma)H_N(\bsigma')] = N\xi(\langle \bsigma,\bsigma'\rangle/N)$. 
 We define the corresponding Gibbs measure at inverse temperature $\beta \ge 0$:
 \begin{align}\label{eq:mu_G}
\mu_{\beta,\bG}(\bsigma) = \frac{1}{Z_N(\beta)} e^{\beta H_N(\bsigma)}\,,~~~\bsigma \in \{-1,+1\}^N\,.
\end{align}

The Gibbs measure $\mu_{\beta,\bG}$ is known to undergo a `static' replica symmetry breaking transition at an inverse temperature $\beta_c$ defined as the largest $\beta$ at which the free energy $\E\log Z_N(\beta)/N$ converges to its annealed value $\beta^2\xi(1)/2+ \log 2$ in the large $N$ limit~\cite{talagrand2000rigorous}. For some mixtures $\xi$ a second `dynamical' transition is expected at a lower inverse temperature $\beta_d$ where the Gibbs measure is expected to exhibit \emph{shattering} between $\beta_d$ and $\beta_c$: a clustering of the mass of the measure into small and well separated subsets, each of exponentially small mass, and which together carry all but an exponentially small fraction of the mass. This transition is also expected to mark the slowdown of natural relaxation dynamics. In the case of the pure $p$-spin model $\xi(x) = x^p$, $p \ge 3$ this dynamical transition was predicted in the statistical physics literature to occur at the value  
\begin{equation}\label{eq:betadyn} 
\beta_d = \inf\left\{\beta >0 : \exists \, q \in (0,1] ~\mbox{s.t.}~ q = \frac{\E\big[\cosh(\beta\sqrt{\xi'(q)} Z)\tanh^2(\beta\sqrt{\xi'(q)} Z)\big]}{\E\big[ \cosh(\beta\sqrt{\xi'(q)} Z)\big]}\right\}\,,
\end{equation}  
where $Z \sim N(0,1)$ via the formalism of the replica method~\cite{montanari2003nature,ferrari2012two}. 
For large values of $p$, $\beta_d$ is asymptotic to $(1+o_p(1)) \sqrt{(2\log p)/p}$ where $o_p(1) \to 0$ as $p \to \infty$; see~\cite{ferrari2012two,alaoui2023samplingpspin}.    
In contrast it is known that $(1-2^{-p})\sqrt{2 \log 2}\le \beta_c \le \sqrt{2 \log 2}$~\cite{talagrand2000rigorous} so the ``shattered phase" is conjectured to occupy a dominant share of the replica-symmetric regime.

\paragraph{Shattering in the Ising $p$-spin.}
Shattering in the Ising pure $p$-spin model was recently established in~\cite{gamarnik2023shattering} for all $\sqrt{\log 2}<\beta <\sqrt{2\log 2}$ and all $p$ sufficiently large. This was shown via first and second moment methods tracking the occurrence of a certain \emph{overlap gap property}. 
In this paper we provide a simple argument showing that shattering occurs for the pure $p$-spin model for all $\beta \in (\sqrt{(2 \log p) /p}, \sqrt{2\log 2})$ when $p$ is sufficiently large, matching the first order term in the asymptotic expansion of $\beta_d$ marking the conjectural dynamical transition: 
\begin{theorem}\label{thm:main1}
Consider the pure $p$-spin case $\xi(x) = x^p$.
For all $\eps>0$, there exists an integer $p_0=p_0(\eps)$ such that for all $p \ge p_0$ and all $(1+\eps)\sqrt{(2 \log p) /p} \le \beta \le (1-\eps)\sqrt{2\log 2}$, there exits a random collection of subsets $\cC_1,\cdots,\cC_m \subset \{-1,+1\}^N$ (measurable with respect to $\bG$) with the following properties: 
\begin{enumerate}
\item Small diameter: $\max_{\bsigma_1,\bsigma_2 \in \cC_i} d_H(\bsigma_1,\bsigma_2) \le rN$ for all $i \le m$.
\item Pairwise separation: $\min_{1\le i \neq j \le m} d_H(\cC_i,\cC_j) \ge R N$. 
\item Small Gibbs mass: $ \E \max_{1\le i\le m} \mu_{\beta,\bG}(\cC_i) \le e^{-cN}$.  
\item Collective coverage of the Gibbs measure: $\E\mu_{\beta,\bG}\big(\bigcup_{i=1}^m \cC_i\big) \ge 1 - e^{-cN}$.
\end{enumerate}
In the above, $d_H$ is the Hamming distance, $c = c(\eps)>0$ and there exists $\delta = \delta(\eps)>0$ such that one can take $r = 1/p^{1+\delta}$ and $R = \delta/p$. 
\end{theorem}
We remark that the scaling of the radii $r$ and $R$ are such that the clusters $\cC_i$ are almost ``point-like" and are at a much larger distance apart than their diameter for large $p$. 

The above theorem is a consequence of a \emph{soft overlap gap property} that appears in a general mixture $\xi$ above an inverse temperature $\bar{\beta}_d$ defined as follows:
Let $\varphi$ and $\Phi$ be respectively the density and cumulative function of the standard normal distribution, and define  
\begin{equation}\label{eq:upperbetadyn}
\bar{\beta}_d = \inf_{q \in (0,1)} \frac{2\sqrt{\xi'(1)}\varphi\big(\Phi^{-1}\big(\frac{1+q}{2}\big)\big)}{\xi(1)-\xi(q)} \, .
\end{equation}  
In the special case of the pure $p$-spin, $\bar{\beta}_d \le (1+o_p(1))\sqrt{(2\log p)/p}$ (see Eq.~\eqref{eq:beta_d0}).

We mention that for the spherical $p$-spin model, the available proofs of shattering rely on different techniques based on the analysis of the stationary points of the TAP free energy~\cite{arous2024shattering} or the monotonicity of the Franz--Parisi potential~\cite{alaoui2023shattering}. We show in Section~\ref{sec:spherical} that the argument used to prove Theorem~\ref{thm:main1} can also be used in the spherical case. This would provide an alternative proof of the shattering result of~\cite{alaoui2023shattering} with an explicit lower bound on $\beta$. In contrast with the Ising case however, the spherical analogue of $\bar{\beta}_d$ is asymptotic to $2.216 > \sqrt{e}$, the latter being the conjectural value of the inverse temperature marking the dynamical transition for large $p$ in the spherical case.

\paragraph{Rarity of solutions found by Lipschitz algorithms.} 
As mentioned earlier, shattering has algorithmic consequences pertaining to hardness of approximate sampling: Glauber dynamics (or Langevin dynamics in the spherical case) slow down exponentially and \emph{stable} algorithms are not able to approximately sample from the Gibbs measure when shattering occurs. This is because shattering implies \emph{transport disorder chaos} of the Gibbs measure--a notion of instability of $\mu_{\beta,\bG}$ in the Wasserstein-2 metric under small perturbations of the disorder $\bG$--which in turn implies that stable algorithms cannot approximate $\mu_{\beta,\bG}$ in this same metric; see~\cite[Section 3]{alaoui2023samplingpspin} and~\cite[Section 5]{alaoui2023shattering} for more detail on this implication.        
A second aim of this paper is to investigate the behavior of \emph{optimization} or \emph{search} algorithms in the shattered phase. Suppose we are interested in efficiently finding a point $\bsigma$ whose energy is approximately typical under $\mu_{\beta,\bG}$. More precisely, for $\beta < \beta_c$ we consider the problem of efficiently returning a point in the super-level set 
 \begin{equation}\label{eq:levelset}
 S_{\beta}(\bG) = \Big\{\bsigma \in \{-1,+1\}^N :  H_N(\bsigma) \ge \beta \xi(1) N\Big\} \, ,
 \end{equation}
 on input $\bG$, $\beta$ and $\xi$. The right-hand side of the inequality in the above definition is the typical energy for a configuration at inverse temperature $\beta$, given by the derivative in $\beta$ of the limiting free energy $\beta^2\xi(1)/2 +\log 2$ for $\beta<\beta_c$. 
 It is not difficult to show that $\mu_{\beta,\bG}$ is exponentially concentrated on $S_{(1-\eps)\beta}(\bG)$ for any $\eps>0$:
\begin{equation}\label{eq:concentration0} 
\E\mu_{\beta,\bG}\big(S_{(1-\eps)\beta}(\bG)\big) \ge 1- e^{-cN}\,,~~~ c = c(\beta,\eps)>0\,,
\end{equation}
heuristically suggesting that shattering of the Gibbs measure must impact the behavior of search algorithms on $S_{(1-\eps)\beta}(\bG)$ (in fact the shattering clusters $(\cC_i)_{1\le i \le m}$ for $\mu_{\beta,\bG}$ of Theorem~\ref{thm:main1} in the pure $p$-spin case all belong to $S_{(1-\eps)\beta}(\bG)$ for some $\eps>0$ small enough, see Section~\ref{sec:shattering}).   
Our main result is this direction is that whenever $\beta > \bar{\beta}_d$ where $\bar{\beta}_d$ is defined in Eq.~\eqref{eq:upperbetadyn}, stable search algorithms must land in an exceptionally rare region of the search space $S_{\beta}(\bG)$, provided that they have a high probability of returning a solution at all.        

 We say that an algorithm $\mA : \R^{M} \to \R^N$ is $L$-Lipschitz if 
 \begin{equation}\label{eq:lipschitzalg}
  \|\mA(\bG) - \mA(\bG')\| \le L \|\bG - \bG'\| \, ,~~~ \forall\, \bG,\bG' \in \R^{M}  \,, 
   \end{equation}
 where $\|\cdot\|$ is the Euclidean norm, and $M = \sum_{k=2}^P N^k$. 
 
 \begin{theorem}\label{thm:main2}
Let $\beta \in(\bar{\beta}_{d}, \beta_c)$, $L > 0$, $B > 0$. There exists $\eps = \eps(\beta,\xi)>0$ such that for all $\beta' \in [(1-\eps)\beta,\beta]$ there exists a subset $E(\bG) \subset S_{\beta'}(\bG)$ depending only on $\beta', L, \xi, \bG$ with exponentially small Gibbs mass relative to $\mu_{\beta,\bG}$:
  \begin{equation}\label{eq:rare_subset}
     \E \mu_{\beta,\bG}(E(\bG)) \le e^{-c_0N}\,,~~~ c_0 = c_0(\beta,L,\xi)>0\,, 
\end{equation}
 such that if $\mA$ is a $L$-Lipschitz algorithm satisfying $\sup_{\bG}\|\mA(\bG)\| \le \sqrt{BN}$ and $\P(\mA(\bG) \notin S_{\beta}(\bG)) \le \eps_0$ for some $\eps_0= \eps_0(B,\beta,\xi)>0$, we have 
 \begin{equation}\label{eq:failure}
     \P\big(\mA(\bG) \in E(\bG)\big) + 4 \P\big(\mA(\bG) \notin S_{\beta'}(\bG)\big) \ge 1- e^{-c_1N}\,,~~~ c_1 = c_1(\beta,L,\xi)>0\,. 
\end{equation}
 \end{theorem}
 
 \begin{remark} 
 We highlight the contrast between Eq.~\eqref{eq:concentration0} and Eq.~\eqref{eq:rare_subset} where $E(\bG)$ is an exceptionally rare subset of $S_{\beta'}(\bG)$ from the point of view of $\mu_{\beta,\bG}$. Moreover if one alters the definition of $S_{\beta}(\bG)$ to be the set of point of energies in the interval  $[(1-\eps)\beta \xi(1) N, (1+\eps)\beta \xi(1) N]$ (a near level set of typical energies), this new set--call it $S_{\beta,\eps}(\bG)$--still verifies the concentration bound~\eqref{eq:concentration0}, and the above theorem still holds with $S_{\beta'}(\bG)$ replaced by $S_{\beta,\eps}(\bG)$ and the bound~\eqref{eq:rare_subset} replaced by a statement about the relative cardinalities: $\E \big[|E(\bG)|/|S_{\beta,\eps}(\bG)|\big] \le e^{-c_0 N}$, $c_0 = c_0(\beta,\eps,L,\xi)>0$. This is because of the fact that $\mu_{\beta,\bG}$ behaves like the uniform measure on $S_{\beta,\eps}(\bG)$ up to multiplicative errors of order $e^{c'N}$ where $c' = c'(\eps) \to 0$ as $\eps \to 0$.    
 \end{remark}                 
  
An inspection of the proof of the above theorem in Section~\ref{sec:exceptional} reveals that Lipschitzness of the algorithm can be relaxed to the property of \emph{overlap concentration}, as already observed in~\cite{huang2025tight} in the context proving lower bounds on optimization as we mention in the next paragraph. The latter property requires that the inner product $\langle \mA(\bG),\mA(\bG')\rangle/N$ where $(\bG,\bG')$ are two correlated copies of the disorder concentrate exponentially in $N$ around its expectation. This is more general than the Lipschitz property and can potentially be the case of algorithms lacking uniform stability in the input.              

The \emph{algorithmic threshold} $E_{\salg}$ of optimizing the Hamiltonian $H_N$ is the maximum energy achievable by a Lipschitz algorithm. It is known that for $E< E_{\salg}$, a Lipschitz algorithm returning a point in $S_{\beta}(\bG)$ with $E=\beta\xi(1)$, exists in the form of incremental Hessian ascent algorithms and incremental approximate message passing~\cite{subag2021following,montanari2021optimization,ams2020,jekel2024potential}, and that if  $E> E_{\salg}$ then no Lipschitz algorithm (with $L = O_N(1)$) can find a point in $S_{\beta}(\bG)$ with probability higher than $e^{-c'N}$, $c' = c'(\beta,L,\xi)>0$~\cite{huang2025tight}. The value of $E_{\salg}$ is given by a stochastic control problem whose dual form is an extended Parisi-type variational formula which is rather difficult to analyze~\cite{ams2020}. We currently do not know how $\bar{\beta}_d\xi(1)$ and $E_{\salg}$ compare for Ising spin glasses, even for the pure $p$-spin model. Our result is most meaningful for mixtures where $ \bar{\beta}_{d}< E{\salg}/\xi(1)$. We comment in Section~\ref{sec:spherical} on the spherical case where $E_{\salg}$ is explicitly known in the pure $p$-spin model.

 \paragraph{A soft overlap gap property.} Theorems~\ref{thm:main1} and~\ref{thm:main2} are both a consequence of a `soft' version of \emph{the overlap gap property} (OGP) which we prove holds for any mixture $\xi$ for all $\beta \in (\bar{\beta}_d,\beta_c)$.  
 In a random optimization problem, the overlap gap property asserts that there exists \emph{no} cluster of points in the solution space (here, $S_{\beta}(\bG)$) at predetermined pairwise distances from each other. Its simplest form asserts the non existence of pairs of points at some fixed distance.   
 This concept was first used to rule out the existence of some class of efficient algorithms for finding large independent sets on a random regular graph~\cite{gamarnik2014limits}. It was then extended to many other random optimization problems where this technique currently provides the best computational lower bounds against large families of algorithms~\cite{gamarnik2017performance,gamarnik2021circuit,gamarnik2020optimization}; for instance a lower bound on Lipschitz algorithms at the algorithmic threshold $E_{\salg}$ discussed above was obtained via a sophisticated version of this argument called the \emph{branching} OGP~\cite{huang2025tight}. We refer the reader to~\cite{gamarnik2021survey} for a survey and further references.          

 In the Ising pure $p$-spin model with large $p$ it is known that the OGP for pairs of points undergoes a sharp phase transition in $S_{\beta}(\bG)$ where it is present with high probability for energies $\beta > \sqrt{\log 2}$ and absent with high probability if $\beta < \sqrt{\log 2}$~\cite{gamarnik2023shattering,kizildaug2023sharp}. (More generally the more sophisticated version of ``$m$-OGP" undergoes a sharp transition at $\beta_m = \sqrt{(2 \log 2)/m}$.) This is precisely the reason the shattering argument used in~\cite{gamarnik2023shattering} is limited to the range $(\sqrt{\log 2},\sqrt{2\log 2})$.  
 Our main technical tool is the observation that the following relaxed (or soft) version of OGP holds all the way down to $\bar{\beta}_d$: a \emph{typical} point from $\mu_{\beta,\bG}$ will have no point of \emph{typical} energy (or more) at some prescribed distance from it, with high probability:                 
   
\begin{theorem}\label{thm:soft_ogp0}
 For $\bar{\beta}_{d} < \beta < \beta_c$, there exists $0<\lbq<\ubq <1$ depending on $\beta$ and $\xi$ such that  
  \begin{equation}\label{eq:soft_ogp0}
     \E \mu_{\beta,\bG}\Big(\Big\{\bsigma : \exists\, \bsigma' \,\textup{s.t.}\, H_N(\bsigma') \ge \beta \xi(1) N ,\,  \lbq \le \langle \bsigma,\bsigma'\rangle/N  \le \ubq \Big\}\Big) \le e^{-cN}\,, 
\end{equation}
where $c = c(\beta,\xi)> 0$. 
\end{theorem}
We remark that the above theorem also holds if $\bG$ is replaced by a correlated copy $\bG'$ of $\bG$ in the construction of $H_N$ in the above event. 
 The bound~\eqref{eq:soft_ogp0} also holds if the energy lower bound on $H_N(\bsigma')$ is taken at an inverse temperature $\beta'$ slightly lower than $\beta$; see Theorem~\ref{thm:soft_ogp}.   
While OGP can be used to rule out the possibility of finding solutions using Lipschitz algorithms, a soft OGP can be used to show that the solutions returned by Lipschitz algorithms are confined to an exceptional set.  This is the idea behind the proof of  Theorem~\ref{thm:main2}; see Section~\ref{sec:local}.   

We finally mention that a statement similar to Theorem~\ref{thm:soft_ogp0} was shown to hold for the symmetric binary perceptron model with $\ubq = 1 - 1/N$--meaning $\bsigma' \neq \bsigma$--and with a weaker bound on the right-hand side, proving that most solutions are isolated up to linear distance~\cite{perkins2021frozen,abbe2022proof}. This was used in~\cite{alaoui2024hardness} to show that the uniform measure on the set of solutions exhibits transport disorder chaos which as mentioned earlier implies failure of stable sampling algorithms. Furthermore it was shown in~\cite{abbe2022binary} that a certain multiscale majority vote algorithm is able to find solutions all belonging to a connected cluster of diameter linear in $N$ at a sufficiently small constraint-to-variable density, showing that this algorithm finds rare solutions in this low density regime. A version of Theorem~\ref{thm:main2} can be proved for the symmetric binary perceptron via the same argument where (given current technology) the size of the exceptional set $E(\bG)$ and the failure probability of Eq.~\eqref{eq:failure} are both $o_N(1)$ instead of exponentially small. This implies that this rareness property is shared by all Lipschitz (more generally, overlap-concentrated) algorithms at sufficiently small density.

\section{Preliminaries}
\label{sec:prelim}

The proof of Theorems  \ref{thm:main1} and \ref{thm:soft_ogp} will rely on studying the local landscape around a typical point from the Gibbs measure. This is easiest done in a planted model and then transferred by contiguity to the original model.

In the following we describe a joint distribution $\nu_{\pl}$ of a triplet $(\bG,\bG_\tau, \bsigma)$ which we refer to as the \emph{planted model}. Let $\bsigma$ be drawn uniformly from $\{-1,+1\}^N$ and we construct the collection of disorder random variables $\bG$ as follows for all $k \ge 2$, $1 \le i_1,\cdots,i_k \le N$,
\begin{equation}\label{eq:dis}
g_{i_1,\cdots,i_k} = \frac{\beta \gamma_k}{N^{(k-1)/2}} \sigma_{i_1}\cdots \sigma_{i_k} + \tilde{g}_{i_1,\cdots,i_k}\,, 
\end{equation}
where $\tilde{g}_{i_1,\cdots,i_k} \sim N(0,1)$ independently of $\bsigma$ and of each other. Next for $\tau \in [0,1]$ we let 
\begin{equation}\label{eq:interp}
\bG_\tau = (1-\tau) \bG + \sqrt{2\tau-\tau^2}\bG' \,,
\end{equation}
where $\bG'$ is an independent collection of i.i.d.\ $N(0,1)$ r.v.'s. 

On the other hand we let $\nu_{\rd}$ (the \emph{null} or \emph{random} model) be the joint law of $(\bG,\bG_\tau,\bsigma)$ where $g_{i_1,\cdots,i_k} \sim N(0,1)$ i.i.d., $\bG_\tau$ is defined as above conditionally on $\bG$, and $\bsigma \sim \mu_{\beta,\bG}$. Let $\Omega$ be the common sample space on which $\nu_{\rd}$ and $\nu_{\pl}$ are defined. A standard fact is that $\nu_{\pl}$ and $\nu_{\pl}$ are ``equivalent" or \emph{contiguous} at the exponential scale for all $\beta<\beta_c$: 

\begin{lemma} \label{lem:contig0}
Let $\beta < \beta_c$, let $(E_N)$ sequence of events defined on $\Omega$. If there exists $c>0, N_0\ge 1$ such that $\nu_{\pl}(E_N) \le e^{-cN}$ for all $N \ge N_0$, then  there exists $c'>0, N_0'\ge 1$ such that $\nu_{\rd}(E_N) \le e^{-c'N}$ for all $N \ge N_0'$.
\end{lemma}
The proof appears in~\cite[Lemma 3.5]{alaoui2023shattering} and goes as follows:  An application of the Bayes rule yields a formula for the likelihood ratio of $\nu_{\pl}$ on $\nu_{\rd}$:
\begin{equation}\label{eq:lr} 
\frac{\rmd \nu_{\pl}}{\rmd \nu_{\rd}} (\bG,\bG_\tau,\bsigma) = \frac{1}{2^N} \sum_{\bsigma' \in \{-1,+1\}^N} e^{\beta H_N(\bsigma') - \beta^2 N \xi(1)/2} \,,
\end{equation}
where the right-hand side does not depend on $\bG_\tau$ nor $\bsigma$, and $H_N$ is the Hamiltonian with disorder coefficients given by $\bG$. 
We further have the concentration bound
\begin{equation} \label{eq:concent}
\P_{\nu_{\rd}}\Big( \frac{1}{N}\Big|\log \frac{\rmd \nu_{\pl}}{\rmd \nu_{\rd}}\Big| \ge t \Big) \le e^{-\beta^2\xi(1)(t-o_N(1))^2N/2} \,,
\end{equation}
for all $t\ge0$, $\beta \le \beta_c$ and all $N$ sufficiently large, by Gaussian concentration of Lipschitz functions~\cite[Theorem 5.6]{boucheron2013concentration}: this is since $\log \big(\rmd \nu_{\pl}/\rmd \nu_{\rd}\big)$ is a Lipschitz function of $\bG$ with Lipschitz constant $\beta\sqrt{\xi(1)N}$, and $\frac{1}{N} \E_{\nu_{\rd}}\log \big(\rmd \nu_{\pl}/\rmd \nu_{\rd}\big) \to 0$ and all $\beta < \beta_c$. Now for a sequence of events $(E_N)$ we have
\begin{align*}
\nu_{\rd}(E_N)  &= \E_{\nu_{\pl}}\Big[ \frac{\rmd \nu_{\rd}}{\rmd \nu_{\pl}} \one_{E_N}\Big] \\
&\le e^{N t} \nu_{\rd}(E_N) + \P_{\nu_{\rd}}\Big(\frac{\rmd \nu_{\pl}}{\rmd \nu_{\rd}} \le e^{-Nt}\Big) \,.
\end{align*}
If $\nu_{\rd}(E_N) \le e^{-cN}$ then we conclude by letting $t=c/2$ and using the bound Eq.~\eqref{eq:concent}.

A first application of this Lemma is to prove that energies near $\beta \xi(1)N$ are exponentially likely under $\mu_{\beta,\bG}$:  

\begin{lemma}\label{lem:typical}
For all $\beta<\beta_c$ and $\eps>0$, 
\[\E \mu_{\beta,\bG} \Big(\Big\{\bsigma :   \Big|\frac{H_N(\bsigma)}{N\xi(1)\beta} -1\Big| \le \eps \Big\}\Big) \ge 1-e^{-cN}\,,\] 
where $c = c(\beta,\eps)>0$. In particular $\E \mu_{\beta,\bG}(S_{(1-\eps)\beta}(\bG)) \ge 1-e^{-cN}$.
\end{lemma}
\begin{proof}
Under $\nu_{\pl}$ we have for any $\bsigma'$, $H_N(\bsigma') = \beta N \xi(\langle \bsigma,\bsigma'\rangle/N) + \tilde{H}_N(\bsigma')$, where $\tilde{H}_N$ is the random Hamiltonian constructed with disorder coefficients $(\tilde{g}_{i_1,\cdots,i_k})$, Eq.~\eqref{eq:dis}. Taking $\bsigma' = \bsigma$ we have
\[\E_{\nu_{\pl}} \mu_{\beta,\bG} \Big(\Big\{\bsigma :   \Big|\frac{H_N(\bsigma)}{N\xi(1)\beta} -1\Big| > \eps \Big\}\Big) 
= \P\big( |\tilde{H}_N(\bsigma)| > N\beta \xi(1)\eps\big)\,.\]
Since $\tilde{H}_N(\bsigma)\sim N(0,\sqrt{N\xi(1)})$ the above is upper-bounded by $2e^{- N \beta^2 \xi(1)\eps^2/2}$.  
We conclude by appealing to Lemma~\ref{lem:contig0}.
\end{proof}

\section{Local landscape around a typical point}
\label{sec:local}
This section is dedicated to the proof of a more general version of Theorem~\ref{thm:soft_ogp0}. 
 Recall that $\bG_\tau = (1-\tau) \bG + \sqrt{2\tau-\tau^2}\bG'$ for $\tau \in [0,1]$ where $\bG'$ is a collection of i.i.d.\ $N(0,1)$ r.v.'s. 
\begin{theorem}\label{thm:soft_ogp}
If $\beta \in (\bar{\beta}_{d}, \beta_c)$, there exists $\eps>0$, $c>0$ and $0\le \lbq<\ubq \le 1$ depending on $\beta$ and $\xi$ such that for all $\tau \in [0,1]$, 
  \begin{equation}\label{eq:soft_ogp}
     \E \mu_{\beta,\bG}\Big(\Big\{\bsigma : \exists\, \bsigma' \in S_{(1-\eps)\beta}(\bG_\tau),\,\,\textup{s.t.}\,\,  \lbq \le \langle \bsigma,\bsigma'\rangle/N  \le \ubq \Big\}\Big) \le e^{-cN}\,.
\end{equation}
In the case $\tau = 1$, if $\beta \in (\bar{\beta}_{d}, \beta_c)$ then for any $\bsigma \in \{-1,+1\}^N$,
\begin{equation}\label{eq:tauequal1}
\P\Big( \exists\, \bsigma' \in S_{\beta}(\bG'),\,\,\textup{s.t.}\,\,  \langle \bsigma,\bsigma'\rangle/N  \ge \lbq \Big) \le e^{-cN}\,.
\end{equation} 

Finally in the case of the pure $p$-spin $\xi(x)=x^p$, for $\eps'>0$, if $\beta \in \big((1+\eps')\sqrt{(2\log p)/p},\\ \sqrt{2\log 2} \big)$, there exists $p_0$, $\delta>0$ and $c'>0$ all depending only on $\eps'$ such that for all $p \ge p_0$, the bounds~\eqref{eq:soft_ogp} and~\eqref{eq:tauequal1} hold with $\lbq = 1 - \delta/p$, $\ubq = 1- 1/p^{1+\delta}$ and $\eps = c'\sqrt{(\log p)/ p}$. 
\end{theorem}
 \begin{proof}
 We proceed by establishing the above result under the planted distribution and then transfer it to the null distribution using contiguity at exponential scale. Fix $q \in \Z_{+}/N$, $\eps>0$ and $\tau\in[0,1]$. Let us also make the dependence of $H_N$ on the disorder explicit by writing $H_{N}(\,\cdot\,; \bG)$ and consider the event  
 \begin{equation}\label{eq:event}
\hspace{-0.1cm}  (\bG,\bG_\tau,\bsigma) : 
\max \Big\{ \frac{1}{N} H_N(\bsigma'; \bG_\tau) : \bsigma' \in \{-1,+1\}^N,  \langle \bsigma,\bsigma'\rangle = Nq \Big\} \ge (1-\eps)\beta \xi(1)\,.
 \end{equation}
 Under $(\bG,\bG_\tau,\bsigma) \sim \nu_{\pl}$ we have $(H_N(\bsigma'; \bG_{\tau}))_{\bsigma'} \stackrel{\rmd}{=} \big(\beta (1-\tau)N \xi(\langle \bsigma,\bsigma'\rangle/N) + \tilde{H}_N(\bsigma')\big)_{\bsigma'}$ where $\tilde{H}_N = H_N(\,\cdot\,;\tilde{\bG})$ is the random Hamiltonian constructed with i.i.d.\ $N(0,1)$ disorder coefficients $(\tilde{g}_{i_1,\cdots,i_p})_{1 \le i_k \le N, k\ge 2}$. In the case $\tau=1$ we need to bound the maximum of $\tilde{H}_N(\bsigma')$  under the constraint $\langle \bsigma,\bsigma'\rangle = Nq$. If $\tau <1$,
we observe that due the constraint $\langle \bsigma,\bsigma'\rangle = Nq \ge 0$ we have $H_N(\bsigma'; \bG_{\tau}) \le H_N(\bsigma'; \bG)$, and it suffices to upper bound the probability of the event displayed in Eq.~\eqref{eq:event} for $\tau=0$. We only consider the latter case, as the former can be treated similarly. We call this event $A_q = A(\bG,\bsigma,q)$. We have 
 \begin{equation}\label{eq:firstbd}
 \E_{\nu_{\pl}} \mu_{\beta,\bG} (A_q) = \P\Big(   \beta  N\xi(q) + \max_{\bsigma' :  \langle \bsigma,\bsigma'\rangle = Nq }  \tilde{H}_N(\bsigma')  \ge (1-\eps)\beta \xi(1) N \Big) \,.
 \end{equation}  
By sign-invariance of the Gaussian distribution we may assume that $\bsigma = \one$. First, since the Hamiltonian $\tilde{H}(\bsigma')$ is $\sqrt{N\xi(1)}$-Lipschitz in the disorder coefficients $(\tilde{g}_{i_1,\cdots,i_p})$ we have
$\P(X_N \ge \E X_N  + t) \le e^{-t^2/(2N\xi(1))}$ for all $t \ge 0$ where $X_N := \max_{\bsigma' :  \langle \one,\bsigma'\rangle = Nq } \tilde{H}_N(\bsigma')$.
We now bound the expectation of $X_N$. The collection $(\tilde{H}_N(\bsigma'))_{\bsigma' \in \{-1,+1\}^N}$ is a centered Gaussian process whose increment variances can bounded as 
 \begin{align}
  \E\big[(\tilde{H}_N(\bsigma')- \tilde{H}_N(\bsigma''))^2\big] &= 2N \big(\xi(1) - \xi( \langle \bsigma',\bsigma''\rangle/N)\big) \nonumber\\
  &\le 2N \xi'(1) \big(1 -  \langle \bsigma',\bsigma''\rangle/N\big)\nonumber\\
  &= \xi'(1)   \E\big[( \langle \bg , \bsigma'\rangle - \langle \bg , \bsigma''\rangle)^2\big]\,,\label{eq:increments}
   \end{align}  
where $\bg \in \R^N$ is a vector of i.i.d.\ $N(0,1)$ r.v.'s.  
It follows by the Sudakov--Fernique inequality~\cite[Theorem 7.2.11]{vershynin2018high} that 
 \begin{align}
 \E \max_{\bsigma' :  \langle \one,\bsigma'\rangle = Nq }  \tilde{H}_N(\bsigma')  &\le \sqrt{\xi'(1)} \E \max_{\bsigma' :  \langle \one,\bsigma'\rangle = Nq }  \langle \bg , \bsigma'\rangle \label{eq:sudakov_fermique}\\
&\le \sqrt{\xi'(1)}  \inf_{h \in \R}  \E \Big[\max_{\bsigma' \in \{-1,+1\}^N}  \langle \bg , \bsigma'\rangle + h  (\langle \one , \bsigma'\rangle- Nq)\Big] \nonumber\\
&=\sqrt{\xi'(1)}  \inf_{h \in \R}  \E \Big[ \sum_{i=1}^N |g_i + h| - Nhq \Big] \nonumber\\
&= \sqrt{\xi'(1)} N  \inf_{h \in \R}  \big\{\E |g_1 + h| - hq \big\}\,.\nonumber
\end{align}
Since 
\[u(h) :=\E |g_1 + h| = \int_{-h}^{+\infty} (z+h)\varphi(z) \rmd z - \int_{-\infty}^{-h} (z+h)\varphi(z) \rmd z = h (2\Phi(h)-1) + 2\varphi(h)\,,\] 
the above infimum is achieved at a point such that $u'(h) = 2\Phi(h)-1 = q$, i.e., $h = \Phi^{-1}((1+q)/2)$, and we obtain the bound 
\begin{equation} \label{eq:sf}
\E \max_{\bsigma' :  \langle \one,\bsigma'\rangle = Nq } \tilde{H}_N(\bsigma') \le 2 \sqrt{\xi'(1)} N \varphi\big(\Phi^{-1}\big(\frac{1+q}{2}\big)\big)\,.
\end{equation}
Therefore recalling the expression of $\bar{\beta}_{d}$ form Eq.~\eqref{eq:upperbetadyn}, if $\beta > \bar{\beta}_d$, there exists an interval $[\lbq,\ubq] \subset [0,1]$ depending on $\beta,\xi$ such that
\begin{equation}\label{eq:condnum}
2\sqrt{\xi'(1)}\varphi\big(\Phi^{-1}\big(\frac{1+q}{2}\big)\big) < \beta(\xi(1)-\xi(q)) \, ,~~\forall q \in [\lbq,\ubq]\,,
\end{equation}
by continuity of the functions appearing in Eq.~\eqref{eq:condnum}. 
Furthermore by continuity we can find $\eps=\eps(\beta,\xi)>0$ such that 
\begin{equation}\label{eq:condnum2}
2\sqrt{\xi'(1)}\varphi\big(\Phi^{-1}\big(\frac{1+q}{2}\big)\big) < (1-\eps)\beta\xi(1) - \beta \xi(q)\, ,~~\forall q \in [\lbq,\ubq]\,.
\end{equation}
By a union bound over $q \in (\lbq,\ubq) \cap \Z_+/N$ (there are at most $N$ points in this set) we obtain from~\eqref{eq:firstbd}  
\[ \E_{\nu_{\pl}} \mu_{\beta,\bG} \Big(\bigcup_{q \in (\lbq,\ubq) \cap \Z_+/N} A_q\Big)  \le Ne^{-c N} \, ,~~~ c =c(\beta,\xi)>0\,.\]
We are able to conclude the argument for a general mixture $\xi$ by invoking contiguity at exponential scale between $\nu_{\pl}$ and $\nu_{\rd}$, Lemma~\ref{lem:contig0}. 

For the pure $p$-spin case we are able to obtain a better quantitative control on the parameters $\lbq, \ubq$ from Eq.~\eqref{eq:condnum}. For $q \in (0,1)$ we let $h = \Phi^{-1}((1+q)/2)$. By the standard bound
\[1-\Phi(h) \ge \varphi(h) \big(\frac{1}{h} - \frac{1}{h^3}\big)\,,\] 
valid for all $h>0$ (see, e.g.,~\cite[Proposition 2.1.2]{vershynin2018high}) we have
\[\varphi\big(\Phi^{-1}\big(\frac{1+q}{2}\big)\big) = \varphi(h) \le \frac{(1-\Phi(h))h}{1-h^{-2}} = \frac{1-q}{2} \frac{h}{1-h^{-2}}\, .\]
We now let $q = 1 - \lambda/p$. Since we have $1-q^p \ge 1-e^{-\lambda}$ for all $\lambda\in [0,p]$, a sufficient condition for  Eq.~\eqref{eq:condnum} to hold for this value of $q$ is
\[ \frac{1}{\sqrt{p}}   \frac{\lambda h}{1-h^{-2}} \le \beta (1-e^{-\lambda})\,.\]
Now we control the magnitude of $h$. We have $\lambda/(2p) = (1-q)/2 = 1-\Phi(h) \le \varphi(h)$. Solving for $h$ we obtain 
\[h \le \sqrt{2 \log( c_0 p/\lambda)} \le \sqrt{2 \log(p/\lambda)}\,,\] 
where $c_0=\sqrt{2/\pi}<1$.
Moreover, since $h = \Phi^{-1}(1-\lambda/(2p)) \to \infty $ as $p \to \infty$ for any $\lambda = o(p)$, we have $ 1/(1-h^{-2}) = 1+o_p(1)$.
Putting these bounds together it suffices to have
\begin{equation}\label{eq:condnum3}
(1+o_p(1)) \frac{\lambda}{1-e^{-\lambda}}  \sqrt{\frac{2 \log(p/\lambda)}{p}} < \beta \, ,
\end{equation}
for some $\lambda = o(p)$. Now let $\beta \ge (1+\eps')\sqrt{(2\log p)/p}$ with $\eps'>0$.  Since the function $v(\lambda) := \lambda/(1-e^{-\lambda})$ is increasing with $v(0^+) = 1$ and $v(+\infty)=+\infty$, for $\lambda \in [p^{-\delta},\delta]$ for some $\delta>0$ the left-hand side in the above display is smaller than      
\[(1+o_p(1))v(\delta) \sqrt{\frac{2 (1+\delta) \log p}{p}}\,,\]
so Eq.~\eqref{eq:condnum3} is satisfied for any $\delta$ such that $v(\delta)\sqrt{1+\delta} < 1+\eps'$ and all $p \ge p_0(\eps')$.  
Moreover since the inequality is strict there further exists $\eps = \eps(\eps',p)>0$ such that  Eq.~\eqref{eq:condnum2} is satisfied with $\lbq = 1-\delta/p$ and $\ubq = 1-1/p^{1+\delta}$: any $\eps$ such that \[\eps \sqrt{2\log 2}  < [(1+\eps')(1-e^{-\lambda}) - (1+o_p(1))\lambda \sqrt{1+\delta}] \sqrt{(2 \log p) / p}\] 
will do.
This concludes the proof the theorem, and also establishes the bound    
\begin{equation}\label{eq:beta_d0} 
\bar{\beta}_d \sqrt{p} \le (1+o_p(1)) v(\lambda)  \sqrt{2 \log(p/\lambda)} \,,
\end{equation}
valid for any $\lambda = o(p)$, where $\bar{\beta}_d$ is defined in Eq.~\eqref{eq:upperbetadyn}.
\end{proof}

\section{Proof of Theorem~\ref{thm:main1}: Shattering}
\label{sec:shattering}
In this section we construct a shattering decomposition for the Gibbs measure and prove Theorem~\ref{thm:main1}. This construction is similar but relatively simpler than the one used for the spherical case in~\cite{alaoui2023shattering}: it uses the soft OGP Theorem~\ref{thm:soft_ogp} as input instead of a putative monotonicity of the Franz--Parisi potential.  

Fix $\eps'>0$ small, and let $p_0=p_0(\eps')$, $p \ge p_0(\eps')$, $\delta = \delta(\eps')>0$, $c' = c'(\eps')>0$, $\eps = c' \sqrt{(\log p)/p}$ and $\lbq = 1- \delta/p$, $\ubq = 1-1/p^{1+\delta}$ be the parameters appearing in Theorem~\ref{thm:soft_ogp} for the pure $p$-spin model, and let  
\begin{equation}\label{eq:betabounds}
(1+\eps')\sqrt{(2\log p)/p} \le \beta \le (1-\eps')\sqrt{2\log 2} < \beta_c\,.
\end{equation}
We further define the inner and outer radii 
\begin{equation}\label{eq:radii}
r = \frac{1}{2}(1-\ubq) = \frac{1}{2 p^{1+\delta}}\,,~~~~~~ R = \frac{1}{2}(1-\lbq) = \frac{\delta}{2p}\,,
\end{equation}
and remark that for $p \ge p_1$ for some $p_1 = p_1(\eps') \ge p_0$ we have $r < R/3$.

We now consider the set of ``regular" points 
\begin{equation}\label{eq:sreg}
\Sreg = \Big\{\bsigma \in \{-1,+1\}^N: \Big|\frac{H_N(\bsigma)}{N\beta} -1\Big| \le \eps/2 \,,\,\nexists\, \bsigma' \in S_{(1-\eps/2)\beta}(\bG) ~\mbox{s.t.}~ d_N(\bsigma,\bsigma') \in [r,R]\Big\}\,,
\end{equation}
where $d_N$ is the Hamming distance normalized by $N$. (Recall that $S_{\beta}(\bG) = \{\bsigma :  H_N(\bsigma) \ge \beta N\}$.) The definition says that the points of $\Sreg$ must have near typical energy and have no points of at least near typical energy at distance between $r$ and $R$ from them. 

It is clear that the points in $\Sreg$ are clustered: if $\bsigma_1 , \bsigma_2 \in \Sreg$ then either $d_{N}(\bsigma_1 , \bsigma_2) < r$ or $d_{N}(\bsigma_1 , \bsigma_2) > R$.
Now consider the `clusters' 
\begin{equation}
C(\bsigma) = \Sreg \cap B_N(\bsigma, r) \, ,~~~~ \mbox{for}~\bsigma \in \Sreg\,,
\end{equation}
where $B_N(\bsigma,r) = \{\bsigma' \in \{-1,+1\}^N : d_N(\bsigma,\bsigma') \le r\}$.

We observe that for $\bsigma_1 , \bsigma_2 \in \Sreg$, either $C(\bsigma_1) = C(\bsigma_2)$ or $d_N(\bsigma_1,\bsigma_2) > R-r$. Indeed, if $d_N(\bsigma_1,\bsigma_2) \le R-r$ then for any $\bsigma \in C(\bsigma_1)$ we have $d_{N}(\bsigma,\bsigma_2) \le d_{N}(\bsigma_1,\bsigma_2)+d_{N}(\bsigma,\bsigma_1) \le (R-r)+r = R$ and since $\bsigma_2 \in \Sreg$, $d_N(\bsigma,\bsigma_2) < r$ so $\bsigma \in C(\bsigma_2)$, hence $C(\bsigma_1) \subseteq C(\bsigma_2)$. By symmetry we have  $C(\bsigma_1) = C(\bsigma_2)$.   

We can now define an equivalence relation on $\Sreg$: $\bsigma_1 \sim \bsigma_2$ if $C(\bsigma_1) = C(\bsigma_2)$, and define the clusters $(\cC_i)_{i=1}^m$ of the shattering decomposition of $\mu_{\beta,\bG}$ as the collection $(C(\bsigma))_{\bsigma \in \cR}$ where $\cR$ collects one representative from each equivalence class of $\sim$. (Then $m = |\cR|$ is the number of equivalence classes.)  
Now we verify the four conditions defining shattering as per Theorem~\ref{thm:main1}:
\begin{enumerate}

\item Small diameter: By definition, each cluster has $d_N$-diameter at most $2r$. Since the points of each cluster belong to $\Sreg$ and $R > 2r$, the $d_N$-diameter is actually at most $r$.       

\item Pairwise separation: From the clustering property established in the previous paragraph we have 
\begin{equation}\label{eq:separation}
d_N(C(\bsigma_1), C(\bsigma_2)) = \min_{\bsigma \in C(\bsigma_1),\bsigma' \in C(\bsigma_2)}d_N(\bsigma,\bsigma') \ge d_N(\bsigma_1,\bsigma_2) - 2r \,,
\end{equation} 
for any two distinct points $\bsigma_1,\bsigma_2 \in \cR$. Since $d_N(\bsigma_1,\bsigma_2) > R > 3r$ it follows that $d_N(C(\bsigma_1), C(\bsigma_2)) > r$ so actually $d_N(C(\bsigma_1), C(\bsigma_2)) > R$. This establishes the separation condition.

 \item Small Gibbs mass: Let $\bsigma \in \cR$. Since $H_N(\bsigma') \le (1+\eps/2)N\beta$ for any $\bsigma' \in \Sreg$ we have  
\[\mu_{\beta,\bG}(C(\bsigma)) = \frac{1}{Z_N(\beta)} \sum_{\bsigma'\in C(\bsigma)} e^{\beta H_{N}(\bsigma')} 
\le  \frac{|C(\bsigma)|}{Z_N(\beta)} e^{(1+\eps/2)N\beta^2} \,. \]
One the one hand, by Stirling's formula, $|C(\bsigma)|\le |B_N(\bsigma,r)| \le e^{N(1+o_N(1))h(2r)}$, where $h(s) = -s \log s - (1-s)\log (1-s)$.
On the other hand, $\frac{1}{N} \E\log Z_N(\beta) = \beta^2/2+\log 2+ o_N(1)$ for all $\beta< \beta_c$.  
On the event $A = \{\frac{1}{N} \log Z_N(\beta) \ge \beta^2/2+\log 2 - t \}$ for $t>0$ to be chosen later we have
 \begin{align*}
\frac{1}{N} \log \mu_{\beta,\bG}(C(\bsigma)) \le (1+o_N(1))h(2r) + (1+\eps/2) \beta^2 - \beta^2/2 - \log 2 + t\, .
 \end{align*}
Since $\beta \le (1-\eps')\sqrt{2\log 2}$, we take $t = [1-(1-\eps')^2(1+\eps)](\log 2)/2>0$ so that the above is bounded by  
 \[ (1+o_N(1))h(2r)  - [1-(1-\eps')^2(1+\eps)](\log 2)/2\, .\] 
 Next we have $\eps^2 = c'^2 (\log p) / p$, and since $2r = 1/p^{1+\delta}$ (see Eq.~\eqref{eq:radii}) we have $h(2r) = O((\log p)/p^{1+\delta})$,  so that $\max_{\bsigma \in \cR}\frac{1}{N} \log \mu_{\beta,\bG}(C(\bsigma)) < -C\eps'$ under the event $A$, for $p$ large enough and where $C>0$ is some absolute constant.   
Finally since $\bG \mapsto \log Z_N$ is a $\beta\sqrt{N}$-Lipschitz function, by Gaussian concentration of Lipschitz functions, $\P(A) \ge 1 - e^{-cN}$, $c = c(\beta)>0$. This yields $\E \max_{\bsigma \in \cR}  \mu_{\beta,\bG}(C(\bsigma)) \le e^{-c'N}$, $c' = c'(\beta,\eps')>0$.      

 \item Collective coverage: Since the collection $(C(\bsigma))_{\bsigma\in \cR}$ forms a partition of $\Sreg$ we have by a union bound
\begin{align*} 
\E \mu_{\beta,\bG}\Big(\bigcup_{\bsigma \in \cR} C(\bsigma) \Big) &= \E\mu_{\beta,\bG}(\Sreg) \\
&\ge 1-  \E\mu_{\beta,\bG}\big(\big\{\bsigma :  H_N(\bsigma) \notin (1\pm\eps/2) \beta N\big\}\big) \\
&~~~~~ -  \E\mu_{\beta,\bG}\big(\big\{\bsigma : \exists \bsigma' \in S_{(1-\eps/2)\beta}(\bG) ~\mbox{s.t.}~ d_N(\bsigma,\bsigma') \in [r,R]\big\}\big)
\end{align*}

The first term in the above is bounded by $e^{-cN}$, $c=c(\beta,\eps)>0$ by Lemma~\ref{lem:typical}, and similarly for the second term by Theorem~\ref{thm:soft_ogp} applied with $\tau=0$. 
This finishes the proof of shattering.
\end{enumerate}

\begin{remark}\label{rmk:separation}
The requirement $r < R/3$ is only needed to ensure pairwise separation between \emph{all} clusters, see Eq.~\eqref{eq:separation}. If one is satisfied with pairwise separation between \emph{most} clusters---for instance this is enough to prove transport disorder chaos; see~\cite[Theorem 5.1]{alaoui2023shattering}---then one only needs $r <R$. Indeed since $d_N(\cC_i,\cC_j) \le R$ implies $d_N(\bsigma_1,\bsigma_2) < 3r$ for any $\bsigma_1 \in \cC_i, \bsigma_2 \in \cC_j$ we have 
\begin{equation*} \E\sum_{1 \le i < j \le m} \mu_{\beta,\bG}(\cC_i)  \mu_{\beta,\bG}(\cC_j) \one\big\{d_N(\cC_i,\cC_j) \le R \big\} 
\le \E \mu_{\beta,\bG}^{\otimes 2}\big(d_N(\bsigma_1,\bsigma_2) < 3r \big) \le o_N(1) \,,
\end{equation*}
where the last bound holds if $3r < 1/2$ (which holds for $p$ sufficiently large), since $d_N(\bsigma_1,\bsigma_2) < 3r$ means $\langle \bsigma_1,\bsigma_2\rangle/N > 1-6r >0$, and the overlap between two replicas from the Gibbs measure concentrates around zero for $\beta< \beta_c$; see~\cite[Theorem 4]{chen2019phase}.
\end{remark}

\section{Proof of Theorem~\ref{thm:main2}: Construction of the exceptional set}
\label{sec:exceptional}
Let $\beta \in (\bar{\beta}_{d},\beta_c)$, $c,\eps>0$ and $\lbq, \ubq$ be the parameters appearing in Theorem~\ref{thm:soft_ogp} for a general mixture $\xi$. Let $\beta' \in [(1-\eps)\beta,\beta]$ and fix $0 < \delta < (\ubq - \lbq)/2$ and $\tau \in [0,1]$. 
By monotonicity of the left-hand side of Eq.~\eqref{eq:soft_ogp} in $\eps$, the same bound holds with $S_{\beta}(\bG)$ replaced by $S_{\beta'}(\bG)$.
Consider the subset of $S_{\beta'}(\bG)$ defined as
\begin{equation}\label{eq:rare_set}
    E_{\tau}(\bG) = \Big\{\bsigma \in S_{\beta'}(\bG) : \P\big(\exists\, \bsigma' \in S_{\beta'}(\bG_\tau),\,\,\textup{s.t.}\,\,  \langle \bsigma,\bsigma'\rangle/N \in (\lbq,\ubq)\,\big| \bG\big) \ge e^{-cN/2}\Big\}\,.
\end{equation} 
The probability in the above definition is with respect to $\bG'$, where $\bG_\tau = (1-\tau) \bG + \sqrt{2\tau-\tau^2}\bG'$. 

\begin{lemma}\label{lem:almostall}
We have $\E \mu_{\beta,\bG}(E_{\tau}(\bG)) \le e^{-cN/2}$, $c>0$. Moreover, $E_{\tau=1}(\bG) = \emptyset$.
\end{lemma}
\begin{proof}
This follows from Theorem~\ref{thm:soft_ogp}, Eq.~\eqref{eq:soft_ogp}: 
\begin{align*}
\E \mu_{\beta,\bG}(E_{\tau}(\bG)) &= \E \sum_{\bsigma} \one\{\bsigma \in E_{\tau}(\bG)\} \, \mu_{\beta,\bG}(\bsigma) \\
&\le e^{cN/2} \E \sum_{\bsigma} \P\big(\exists \,\bsigma' \in S_{\beta'}(\bG_\tau),\,\,\textup{s.t.}\,\,  \langle \bsigma,\bsigma'\rangle/N \in (\lbq,\ubq)\,\big| \bG\big) \, \mu_{\beta,\bG}(\bsigma)\\
&= e^{cN/2} \E \mu_{\beta,\bG}\Big(\Big\{\bsigma : \exists \,\bsigma' \in S_{\beta'}(\bG_\tau),\,\,\textup{s.t.}\,\, \langle \bsigma,\bsigma'\rangle/N \in (\lbq,\ubq)\Big\}\Big) \\
&\le e^{-cN/2}\,,
\end{align*}
where the second line was obtained via the inequality $\one\{ x \ge a\} \le x/a$, $x,a >0$ and the third line follows by exchanging the order of summation. For $\tau =1$, the we have $\P(\exists \,\bsigma' \in S_{\beta'}(\bG'),\,\textup{s.t.}\, \langle \bsigma,\bsigma'\rangle/N \in (\lbq,\ubq)\,|\, \bG) \le e^{-cN}$ since this probability does not depend on $(\bG, \bsigma)$. 
\end{proof}

We now construct the exceptional set as the union of the sets $E_{\tau}(\bG)$ for ``all" $\tau$: we let $K$ be some large integer depending only on $\lbq,\ubq,L$, and consider 
\begin{equation}\label{eq:exceptional} 
E(\bG) = \bigcup_{k=0}^{K-1} E_{\tau_k}(\bG)\, ,~~~\tau_k = k/K\,.
\end{equation}
From Lemma~\ref{lem:almostall} we have $\E\mu_{\bG,\beta}(E(\bG)) \le K e^{-cN/2}$.

Now we fix a $L$-Lipschitz algorithm $\mA : \R^{M} \to \R^N$ such that $\sup_{\bG} \|\mA(\bG)\|\le \sqrt{BN}$, and define its correlation function  
\begin{equation}
\chi_{N}(\tau) = \frac{1}{N}\E\big\langle \mA(\bG),\mA(\bG_{\tau})\big\rangle \,.
\end{equation}
The following facts can be collected from~\cite[Proposition 3.1 and Proposition 8.2]{huang2025tight}: 
\begin{lemma}\label{lem:chiN}
The function $\chi_N$ is continuous, non-increasing and Lipschitz with constant $L^2$. 
Moreover the following concentration bound holds for all $t\ge0$ and all $\tau \in [0,1]$:
 \[ \P\Big( \Big|\frac{1}{N}\big\langle \mA(\bG),\mA(\bG_{\tau})\big\rangle - \chi_N(\tau)\Big| \ge t\Big) \le 2e^{-Nt^2/8L^2}\,.\]
\end{lemma}
\begin{proof}
    The Lipschitz property of $\chi_N$ does not explicitly appear in~\cite{huang2025tight} but is straightforward to deduce from their calculations: referring to the proof of Proposition 3.1 in their paper, we write $\mA_i : \R^M \to \R$ for each output coordinate of $\mA$ and consider its Hermite decomposition $\mA_i(\,\cdot\,) = \sum_{\alpha \in \N_+^{M}} a_{\alpha} h_{\alpha}(\,\cdot\,)$, $h_{\alpha} = \prod_{j=1}^M h_{\alpha_j}$ where $h_j$ are the standard Hermite polynomials, we have $\chi_N(\tau) =  \sum_{j \ge 0} (1-\tau)^j W_j$, where $W_j = (1/N) \sum_{i=1}^N \sum_{|\alpha|=j} a_{\alpha}^2$. Taking a derivative, $|\chi_N'(\tau)| = \sum_{j \ge 1} j(1-\tau)^{j-1} W_j \le \sum_{j \ge 1} j W_j = (1/N)  \sum_{i=1}^N \E \|\nabla \mA_i(\bG)\|^2 \le L^2$, since by Rademacher's theorem each $\mA_i$ is differentiable almost everywhere with a gradient bounded by $L$ in Euclidean norm.           
\end{proof}

Next, in view of applying the intermediate value theorem we control the values of $\chi_N$ at the endpoints $\tau \in \{0,1\}$.    
Let us consider the events  
\begin{equation}
 A(\bG) = \Big\{\mA(\bG) \in S_{\beta'}(\bG)\Big\}\,, ~~R = \Big\{\nexists \,\bsigma' \in S_{\beta'}(\bG') : \big\langle \mA(\bG),\bsigma'\big\rangle/N \ge \lbq\Big\}\,.
 \end{equation}
By Theorem~\ref{thm:soft_ogp}, Eq.~\eqref{eq:tauequal1}, we have $\P(R^c | \bG) \le e^{-cN}$ if $\mA(\bG) \in \{-1,+1\}^N$. Therefore if $\P\big(A(\bG)^c\big) < \min\{\delta/\ubq, \delta/3B\}$ and $N \ge c^{-1}\log(B/\delta)$ we have
\begin{align}
\chi_N(0) &\ge \ubq \P\big(A(\bG)\big) > \ubq - \delta\,,~~~\mbox{and}\label{eq:tau0lb}\\
\chi_N(1) &\le \lbq + B\P\big(A(\bG)^c \cup A(\bG')^c \cup R^c\big) 
\le\lbq + 2B\P(A(\bG)^c) + Be^{-cN} < \lbq + \delta\,.\label{eq:tau1ub}
\end{align}

Now consider a subdivision of $[0,1]$ into $K+1$ equally spaced points $k/K$, $0 \le k \le K$. By Lemma~\ref{lem:chiN} and the bounds~\eqref{eq:tau0lb} and~\eqref{eq:tau1ub},  
for $K \ge 10 L^2 / (\ubq-\lbq)$, there exists some $k \in [1,K-1]$ such that $\chi_N(k/K) \in (\lbq+\delta,\ubq-\delta)$. 
We now let $\tau = k/K$. In addition to the event $A(\bG)$ and $R$ defined above, we define the events
\begin{align}
    B(\bG,\bG_\tau) &= \Big\{\big|\big\langle \mA(\bG),\mA(\bG_{\tau})\big\rangle/N - \chi_N(\tau)\big| \le \delta\Big\}\,,\\
    F(\bG) &= \Big\{\mA(\bG) \in E_{\tau}(\bG)\Big\}\,.
\end{align}
Since $\chi_N(\tau) \in (\lbq + \delta, \ubq - \delta)$ we have $\big\langle \mA(\bG),\mA(\bG_{\tau})\big\rangle/N \in (\ubq,\lbq)$ under the event $B(\bG,\bG_\tau)$.  
Next, from the definition of the set $E_{\tau}(\bG)$, Eq.~\eqref{eq:rare_set}, under $A(\bG) \cap F(\bG)^c$ we have 
\[\P\big(\exists \, \bsigma' \in S_{\beta'}(\bG_\tau), \langle \mA(\bG),\bsigma'\rangle/N \in (\lbq,\ubq)\,\big|\, \bG\big) \le e^{-cN/2}\, .\]
We use $\bsigma' = \mA(\bG_\tau)$ as a witness for the above event to obtain
\[\one_{A(\bG) \cap F(\bG)^c} \P( A(\bG_{\tau}) \cap B(\bG,\bG_{\tau}) \,|\, \bG) \le e^{-cN/2}\,.\]
We now average the above by further restricting to $A(\bG_\tau) \cap B(\bG,\bG_{\tau})$:
\[\E\Big[\one_{A(\bG) \cap A(\bG_\tau) \cap F(\bG)^c \cap B(\bG,\bG_\tau)} \P\big(A(\bG_{\tau}) \cap B(\bG,\bG_{\tau})\,\big| \bG\big) \Big] \le e^{-cN/2}\, .\]
 Adding $\E\big[\one_{A(\bG)\cap A(\bG_\tau) \cap F(\bG) \cap B(\bG,\bG_\tau)} \P\big(A(\bG_{\tau}) \cap B(\bG,\bG_{\tau})\,\big| \bG\big) \big] $ on both sides we obtain 
\begin{equation}\label{eq:lhs1}
\LHS := \E\Big[\one_{A(\bG)\cap A(\bG_\tau) \cap B(\bG,\bG_\tau)}\P\big(A(\bG_{\tau}) \cap B(\bG,\bG_{\tau})\,\big| \bG\big) \Big] \le e^{-cN/2} + \P(F(\bG))\,.
\end{equation}
Now we focus on lower-bounding the left-hand side.  
By the tower property of expectations we have
\begin{align}
\LHS &\ge \E\Big[\one_{A(\bG)} \P\big(A(\bG_\tau) \cap B(\bG,\bG_\tau) \,\big|\,\bG\big)^2\Big]\nonumber\\
&\ge \P\big(A(\bG) \cap A(\bG_\tau) \cap B(\bG,\bG_\tau) \big)^2\,,\nonumber \\
&\ge  \Big(1-2\P(A(\bG)^c) - \P(B(\bG,\bG_\tau)^c) \Big)^2\,,\label{eq:jensen}
\end{align}
where we used Jensen's inequality to obtain the second line and a union bound to obtain the third.  
From~\eqref{eq:lhs1} and~\eqref{eq:jensen} we obtain
\begin{equation*}
 1-4\P(A(\bG)^c) - 2\P(B(\bG,\bG_\tau)^c)  \le e^{-cN/2} + \P(F(\bG))\,.
\end{equation*}
Since $\P(B(\bG,\bG_\tau)^c) \le 2e^{-N\delta^2/8L^2}$ by Lemma~\ref{lem:chiN}, this completes the proof.

\section{The spherical case}
\label{sec:spherical}
The methods used in the previous sections can be applied to the spherical case as well where the Gibbs measure $\mu_{\beta,\bG}$ is defined on the sphere $\mathbb{S}^{N-1}(\sqrt{N})$ of radius $\sqrt{N}$ via the formula~\eqref{eq:mu_G} interpreted as a density relative to the uniform measure on it. In this case, for a general mixture $\xi$, the inverse temperature $\bar{\beta}_d$ takes the form
\begin{equation} \label{eq:betabar_d_sph}
\bar{\beta}_d^{\tiny{\rm{sph}}} = \inf_{q \in (0,1)} \frac{\sqrt{\xi'(1)(1-q^2)}}{\xi(1)-\xi(q)}\,,
\end{equation}
and all the results established in the previous sections extend to the sphere via the same arguments. 
The above expression for $\bar{\beta}_d^{\tiny{\rm{sph}}}$ can be seen from the bound~\eqref{eq:sudakov_fermique} where the maximization is over the sphere rather than the hypercube.
In the case of the pure $p$-spin model $\xi(x) = x^p$, one can verify that for large $p$,
\begin{equation} 
\bar{\beta}_d^{\tiny{\rm{sph}}}  \xrightarrow[p \to \infty]{} \,\,\inf_{\lambda >0 } \frac{\sqrt{2\lambda}}{1-e^{-\lambda}} = 2.2160...
\end{equation}
where the minimizing $q$ in Eq.~\eqref{eq:betabar_d_sph} is approximately equal to $1-\lambda_*/p$, $\lambda_* = 1.2608...$. 
On the other hand, it is known that  the static replica-symmetry breaking transition happens at $\beta_c^{\tiny{\rm{sph}}} = (1+o_p(1))\sqrt{\log p}$, and it is expected that the dynamical transition happens at   
\begin{equation} 
\beta_d^{\tiny{\rm{sph}}} = \sqrt{\frac{(p-1)^{p-1}}{p(p-2)^{p-2}}} \xrightarrow[p \to \infty]{} \sqrt{e} = 1.6487... \,,
\end{equation}
see~\cite{crisanti1993spherical,castellani2005spin}. 
Therefore the spherical pure $p$-spin model ceases to exhibit the soft overlap gap property of Theorem~\ref{thm:soft_ogp0} strictly within the (conjecturally) shattered phase, and a different argument will be needed to establish shattering in the interval $[\beta_d^{\tiny{\rm{sph}}},\bar{\beta}_d^{\tiny{\rm{sph}}}]$.     

A shattering decomposition for the Gibbs measure was constructed in this spherical case for $\beta \in [C,\beta_c^{\tiny{\rm{sph}}})$ where $C$ is some absolute constant  in~\cite{alaoui2023shattering} by showing that the Franz--Parisi potential is strictly increasing in an interval of the form $q\in [1-\lambda_1/p,1-\lambda_2/p]$ with $\lambda_2 \ll \lambda_1$ for large $p$ (and large $\beta$, which is allowed since $\beta_c^{\tiny{\rm{sph}}}$ is diverging in $p$). This result can be reproved using the methods of this paper with a constant $C$ taking the value  
\begin{equation} 
C = \inf\Big\{v \ge \bar{\beta}_d^{\tiny{\rm{sph}}}  ~:~ \exists\, \lambda_1 > 3 \lambda_2 ~~\mbox{s.t.}~~  \frac{\sqrt{2\lambda}}{1-e^{-\lambda}} \le v \,,\,\,\forall\,\, \lambda \in [\lambda_1,\lambda_2] \Big\} = 2.342... \,,
\end{equation}
with $\lambda_2 = 0.71...$ and $\lambda_1 = 2.13...$ The factor $3$ between $\lambda_1$ and $\lambda_2$ is needed to ensure the pairwise separation between all clusters; see Eq.~\eqref{eq:separation}. If one is satisfied with separation between most clusters then shattering can be proved for all $\beta> \bar{\beta}_d^{\tiny{\rm{sph}}} $; see Remark~\ref{rmk:separation}.
These numerical values can probably be improved by analyzing the true maximum of the Hamiltonian $\tilde{H}_N$ appearing in the proof of Theorem~\ref{thm:soft_ogp} via the Crisanti--Sommers formula on the sphere~\cite{crisanti1992spherical,talagrand2006free}, instead of applying a Gaussian comparison bound, Eq.~\eqref{eq:increments}, but this approach is unlikely to succeed all the way down to $\beta_d^{\tiny{\rm{sph}}}$.  

It is also interesting to compare $\bar{\beta}_d^{\tiny{\rm{sph}}}$ to the algorithmic threshold $E_{\salg}$ for maximizing the Hamiltonian $H_N$. Its value is explicitly known in the spherical case~\cite{subag2021following}:  
\[E_{\salg} = \int_0^1 \sqrt{\xi''(x)} \,\rmd x\,. \]
In the pure $p$-spin case, $E_{\salg} = 2\sqrt{(p-1)/p} = 2(1+o_p(1)) < \bar{\beta}_d^{\tiny{\rm{sph}}}$ for large $p$. Since no Lipschitz algorithm can find a point in $S_{E}(\bG)$ with a probability which not exponentially small for any $E > E_{\salg}$~\cite{huang2025tight}, our Theorem~\ref{thm:main2} is vacuous in this case.

\vspace{5mm}
\textbf{Acknowledgments.} The author would like to thank Mark Sellke for instructive conversations. 
\vspace{5mm}

 \newpage
\providecommand{\bysame}{\leavevmode\hbox to3em{\hrulefill}\thinspace}
\providecommand{\MR}{\relax\ifhmode\unskip\space\fi MR }
\providecommand{\MRhref}[2]{%
  \href{http://www.ams.org/mathscinet-getitem?mr=#1}{#2}
}
\providecommand{\href}[2]{#2}

\end{document}